\documentclass[11pt]{amsart}
\usepackage[margin=1in]{geometry}
\usepackage{amsmath,amssymb,mathtools}
\usepackage[T1]{fontenc}
\usepackage{lmodern}
\usepackage{microtype}
\usepackage{hyperref}
\hypersetup{colorlinks=true,linkcolor=blue,citecolor=blue,urlcolor=blue}

\newtheorem{theorem}{Theorem}
\newtheorem{lemma}[theorem]{Lemma}
\newtheorem{proposition}[theorem]{Proposition}
\newtheorem{corollary}[theorem]{Corollary}
\newtheorem{conjecture}[theorem]{Conjecture}

\newcommand{\F}{\mathcal{F}}
\newcommand{\G}{\mathcal{G}}
\newcommand{\A}{\mathcal{A}}
\newcommand{\B}{\mathcal{B}}

\title{An Intersection-Weighted Erd\H{o}s--Ko--Rado Theorem}
\author{Casey Tompkins}
\email{casey.tompkins@renyi.hu}
\date{}

\begin{document}

\begin{abstract}
We consider an Erd\H{o}s--Ko--Rado type sum that weights each member of a uniform family according to its smallest intersection with the rest of the family. We prove that once the ground set is sufficiently large this sum is at most one, with equality exactly for stars. This simultaneously generalizes the usual Erd\H{o}s--Ko--Rado theorem for every intersection threshold $t$ and $n$ sufficiently large.
As a consequence we also obtain an extension of Hilton's theorem on cross-intersecting families.
\end{abstract}

\maketitle

\section{Introduction}

A family $\A\subseteq \binom{[n]}{k}$ is \emph{$t$-intersecting} if
$|A\cap B|\ge t$ for all $A,B\in \A$. The basic example is a $t$-\emph{star}: the
family of all $k$-sets containing a fixed $t$-set, which has size
$\binom{n-t}{k-t}$. The theorem of Erd\H{o}s, Ko, and Rado \cite{EKR} states that when $t=1$ and $n\ge 2k$, every intersecting
$\A\subseteq \binom{[n]}{k}$ has size at most $\binom{n-1}{k-1}$, with
equality only for $1$-stars when $n>2k$. The same paper also proved that, for
fixed $k$ and $t$, the $\binom{n-t}{k-t}$ bound holds once $n$ is sufficiently large.

The best possible lower bound on $n$ ensuring that a $k$-uniform $t$-intersecting family has size at most $\binom{n-t}{k-t}$ was obtained by Frankl and Wilson. Frankl~\cite{Frankl} proved that the
threshold $n\ge (t+1)(k-t+1)$ is sufficient when $t \ge 15$, and Wilson~\cite{Wilson} proved that it is sufficient for all $t$.  

The full extremal problem for  $k$-uniform $t$-intersecting families for all values
of $n,k,t$ was later solved by Ahlswede and Khachatrian~\cite{AK}. Their
complete intersection theorem shows that the extremal examples are, up to
isomorphism, families of the form
\[
\{A\in \binom{[n]}{k}: |A\cap [t+2r]|\ge t+r\}.
\]
The case $r=0$ is the $t$-star, while the case $r=1$ is the family that appears
below in the sharpness examples.

For a family $\F\subseteq \binom{[n]}{k}$ and $A\in \F$, write
\[
i_{\F}(A):=\min_{B\in \F}|A\cap B|.
\]
We consider strengthening the Erd\H{o}s--Ko--Rado theorem by showing that for sufficiently large $n$ the sum
\[
\sum_{A\in \F}\frac{1}{\binom{n-i_{\F}(A)}{k-i_{\F}(A)}}.
\]
is bounded above by one. 

Our first result gives a short proof that a cubic size condition on the ground-set is enough.

\begin{theorem}\label{thm:cubic}
Let $k\ge 1$, let $\F\subseteq \binom{[n]}{k}$, and assume that
$n\ge (k^3+2k^2+k)/2$. Then
\[
\sum_{A\in \F}\frac{1}{\binom{n-i_{\F}(A)}{k-i_{\F}(A)}}\le 1.
\]
Moreover, equality holds if and only if $\F$ is a $c$-star for some $0\le c\le k$.
\end{theorem}
Our main result improves the ground-set condition to one that is asymptotically best possible in~$k$.

\begin{theorem}\label{thm:quadratic}
There exists an absolute constant $D$ such that the following holds. Let $k\ge 1$, let
$\F\subseteq \binom{[n]}{k}$, and assume that $n\ge k^2/4+5k+D$. Then
\[
\sum_{A\in \F}\frac{1}{\binom{n-i_{\F}(A)}{k-i_{\F}(A)}}\le 1.
\]
Moreover, equality holds if and only if $\F$ is a $c$-star for some $0\le c\le k$.
\end{theorem}

To see the leading $k^2/4$ term in Theorem~\ref{thm:quadratic} is best possible  we consider the Ahlswede--Khachatrian construction with $r=1$:
\[
J_t(n,k):=\{A\in \binom{[n]}{k}: |A\cap [t+2]|\ge t+1\}.
\]

In Section~\ref{sec:sharpness} we show $J_t(n,k)$ violates our inequality whenever $n<k+(k-t)(t+1)$.  Taking $t=\lfloor k/2\rfloor$ gives
counterexamples up to
\[
n<\frac{k^2}{4}+\frac{3k}{2}+O(1),
\]
so the order and the leading coefficient in Theorem~\ref{thm:quadratic} cannot be improved.

For the rest of the paper we denote the sum in Theorems~\ref{thm:cubic} and~\ref{thm:quadratic} by
\[
\Phi_{n,k}(\F):=\sum_{A\in \F}\frac{1}{\binom{n-i_{\F}(A)}{k-i_{\F}(A)}}.
\]
If $\F$ is $t$-intersecting, then $i_{\F}(A)\ge t$ for every $A\in\F$, and hence
\[
\Phi_{n,k}(\F)\ge \frac{|\F|}{\binom{n-t}{k-t}}.
\]
Thus Theorem~\ref{thm:quadratic} simultaneously strengthens the usual $t$-star upper bounds for every
$t$ in its range of $n$.

The idea of replacing a sequence of extremal results by an inequality over arbitrary families has been considered in other settings; in graph terminology this is sometimes referred to as a \emph{localized} version of the extremal problem~\cite{MT-local}.
On the graph side, Brada\v{c}~\cite{Br} and Malec and Tompkins~\cite{MT-local} proved localized versions of Tur\'an's theorem, and versions of the Erd\H{o}s--Gallai theorem, the LYM inequality, and Erd\H{o}s--Ko--Rado were considered in~\cite{MT-local} (this earlier version is weighted based on matching number rather than $t$-intersection). Applications of such results to Ramsey--Tur\'an problems appear in~\cite{BCMM}.
Kirsch and Nir~\cite{KN} considered localized versions of generalized Tur\'an problems. Several further localized or closely related weighted and spectral results have been considered; see for example~\cite{AS,BBL,ZZ-cycles,ZZ-hyperEG,LN,AC-framework,AC-Turan,LiuNing,KKP}.

A related two-part inequality arises from Borg's work~\cite{Borg} (see also Wang and Zhang~\cite{WZ}).
For $\A\subseteq \binom{[n]}{k}$, write
\[
\A^{t,+}:=\{A\in \A: |A\cap B|\ge t \text{ for all } B\in \A\setminus\{A\}\},
\qquad
\A^{t,-}:=\A\setminus \A^{t,+}.
\]
In the range $n\ge (t+1)(k-t+1)$, Borg's result~\cite[Theorem~3.8 and
Corollary~3.9]{Borg}, together with Wilson's result on the largest $t$-intersecting families
on $\binom{[n]}{k}$~\cite{Wilson}, gives
\[
\frac{|\A^{t,+}|}{\binom{n-t}{k-t}}+\frac{|\A^{t,-}|}{\binom{n}{k}}\le 1.
\]
Since $\A^{t,+}=\{A\in \A:i_{\A}(A)\ge t\}$, Borg's inequality is only sensitive to whether $i_{\A}(A)$ is below $t$ or at least $t$. The sum $\Phi_{n,k}$ refines this result to the whole profile of possible intersections
$0,1,\dots,k$.

Two families $\A,\B\subseteq \binom{[n]}{k}$ are
\emph{cross-$t$-intersecting} if $|A\cap B|\ge t$ for every $A\in \A$ and $B\in \B$. When $t=1$ we simply say that the families are \emph{cross-intersecting}. A classical theorem of Hilton~\cite{HiltonCross} shows that for $n\ge 2k$ and given pairwise cross-intersecting families $\F_1,\F_2,\dots,\F_m \subseteq \binom{[n]}{k}$, the sum $\sum_{i=1}^m |\F_i|$ is at most the maximum of $\binom{n}{k}$ and $m\binom{n-1}{k-1}$. The sharp examples are given by either taking one family to be $\binom{[n]}{k}$ and the rest empty, or by taking them all to be identical stars. Borg's work~\cite{Borg} implies that the analogous result holds for cross-$t$-intersecting families.
We prove an extension of Hilton's theorem (and Borg's cross-$t$-intersecting version) to a setting in which the required cross-intersection thresholds are allowed to vary.

\begin{theorem}\label{thm:multi-hilton}
For each $i\in\{1,\dots,k\}$, let $m_i\ge 0$ be integers and let
$\F_{i,j}\subseteq \binom{[n]}{k}$ for $1\le j\le m_i$. Suppose that any two distinct families
$\F_{i,j}$ and $\F_{i',j'}$ are cross-$\max\{i,i'\}$-intersecting. Put
\[
M_s:=\max\!\left\{1,\sum_{r=1}^s m_r\right\}\qquad (0\le s\le k).
\]
If $n$ satisfies the lower bound in either Theorem~\ref{thm:cubic} or Theorem~\ref{thm:quadratic}, then
\[
\sum_{i=1}^k\sum_{j=1}^{m_i}|\F_{i,j}|
\le
\max_{0\le s\le k}M_s\binom{n-s}{k-s}.
\]
\end{theorem}

Restricting to the case when $\F_{i,j}$ are empty except for $i=t$, Theorem~\ref{thm:multi-hilton} yields the bound
\[
\sum_{j=1}^m |\F_{t,j}|
\le
\max\!\left\{\binom{n}{k},m\binom{n-t}{k-t}\right\}
\]
in the large ground-set range covered by Theorems~\ref{thm:cubic} and~\ref{thm:quadratic}. For
$t=1$ this recovers Hilton's cross-intersection theorem \cite{HiltonCross} in that range.

The paper is organized as follows. Section~\ref{sec:prelim} contains lemmas which are common to the proofs of both the cubic and quadratic bounds. Section~\ref{sec:inputs} records the
existing intersection theorems used in the proofs. Section~\ref{sec:cubic} proves the cubic bound, and Section~\ref{sec:quadratic} proves the quadratic bound. Section~\ref{sec:sharpness} gives the quadratic-order
counterexamples, and Section~\ref{sec:hilton} proves some Hilton-type consequences.

\section{Preliminaries}\label{sec:prelim}

The following lemmas are used in both Theorem~\ref{thm:cubic} and Theorem~\ref{thm:quadratic}.

\begin{lemma}\label{lem:core}
Let $\F\subseteq \binom{[n]}{k}$, and write
\[
C:=\bigcap_{A\in \F}A,
\qquad
c:=|C|,
\qquad
\F':=\{A\setminus C:A\in \F\},
\]
viewed as a family of $(k-c)$-subsets of an $(n-c)$-element ground set. Then
\[
\Phi_{n,k}(\F)=\Phi_{n-c,k-c}(\F').
\]
\end{lemma}

\begin{proof}
For each $A\in \F$, the set $A\setminus C$ lies in $\F'$. Since every member of $\F$ contains $C$,
we have
\[
i_{\F}(A)=c+i_{\F'}(A\setminus C).
\]
Therefore
\[
\binom{n-i_{\F}(A)}{k-i_{\F}(A)}
=
\binom{(n-c)-i_{\F'}(A\setminus C)}{(k-c)-i_{\F'}(A\setminus C)}.
\]
Summing over $A\in \F$ proves the lemma.
\end{proof}
We refer to the family $C$ in Lemma~\ref{lem:core} as the \emph{common core} of $\F$. Assume now that $k\ge 1$ and that $\bigcap_{A\in \F}A=\varnothing$. For $0\le t\le k$, define
\[
\G_t:=\{A\in \F:i_{\F}(A)\ge t\}.
\]
Then $\G_0=\F$, each $\G_t$ is $t$-intersecting, and $\G_k=\varnothing$. Indeed, if some
$A\in \G_k$, then $|A\cap B|\ge k$ for every $B\in \F$, so every member of $\F$ equals $A$,
contradicting the assumption that the common core is empty.

\begin{lemma}\label{lem:tel}
Assume that $k\ge 1$ and $\bigcap_{A\in \F}A=\varnothing$. Then
\[
\Phi_{n,k}(\F)
=
\frac{|\F|}{\binom{n}{k}}+
\sum_{t=1}^{k-1}\frac{n-k}{n-t+1}\cdot \frac{|\G_t|}{\binom{n-t}{k-t}}.
\]
In particular,
\[
\Phi_{n,k}(\F)
\le
\frac{|\F|}{\binom{n}{k}}+
\sum_{t=1}^{k-1}\frac{|\G_t|}{\binom{n-t}{k-t}}.
\]
\end{lemma}

\begin{proof}
The members of $\F$ with $i_{\F}(A)=t$ are exactly $\G_t\setminus \G_{t+1}$, so
\[
\Phi_{n,k}(\F)=\sum_{t=0}^{k-1}(|\G_t|-|\G_{t+1}|)\binom{n-t}{k-t}^{-1}.
\]
Since $\G_k=\varnothing$, this telescopes to
\[
\Phi_{n,k}(\F)=\frac{|\F|}{\binom{n}{k}}+
\sum_{t=1}^{k-1}|\G_t|\left(\binom{n-t}{k-t}^{-1}-\binom{n-t+1}{k-t+1}^{-1}\right).
\]
Using
\[
\binom{n-t+1}{k-t+1}^{-1}
=
\frac{k-t+1}{n-t+1}\binom{n-t}{k-t}^{-1},
\]
we obtain
\[
\binom{n-t}{k-t}^{-1}-\binom{n-t+1}{k-t+1}^{-1}
=
\frac{n-k}{n-t+1}\binom{n-t}{k-t}^{-1},
\]
which gives the formula.
\end{proof}

\section{Intersection theorems}\label{sec:inputs}
A $t$-intersecting family is called \emph{trivial} if all its members contain a fixed $t$-set and
\emph{nontrivial} otherwise. We use the following form of the complete nontrivial-intersection theorem of
Ahlswede and Khachatrian \cite{AK-nt}.

\begin{theorem}[Ahlswede--Khachatrian]\label{thm:AK}
Let $1\le t<k$, and let $\A\subseteq \binom{[n]}{k}$ be a nontrivial $t$-intersecting family. If
\[
n>(t+1)(k-t+1),
\]
then
\[
|\A|\le \max\{|H_1(n,k,t)|,|H_2(n,k,t)|\},
\]
where
\[
H_1(n,k,t):=\left\{A\in \binom{[n]}{k}: |A\cap [t+2]|\ge t+1\right\},
\]
and
\[
H_2(n,k,t):=\left\{A\in \binom{[n]}{k}: [t]\subseteq A,\ A\cap [t+1,k+1]\neq \varnothing\right\}
\cup
\left\{[k+1]\setminus\{i\}: i\in [t]\right\}.
\]
Moreover,
\[
|H_1(n,k,t)|=(t+2)\binom{n-t-2}{k-t-1}+\binom{n-t-2}{k-t-2},
\]
and
\[
|H_2(n,k,t)|=\binom{n-t}{k-t}-\binom{n-k-1}{k-t}+t.
\]
\end{theorem}

\begin{corollary}\label{cor:coarse}
Let $1\le t<k$, and let $\A\subseteq \binom{[n]}{k}$ be a nontrivial $t$-intersecting family. If
\[
n>(t+1)(k-t+1),
\]
then
\[
|\A|\le (k+1)\binom{n-t-1}{k-t-1}.
\]
\end{corollary}

\begin{proof}
By Theorem~\ref{thm:AK}, it is enough to bound the sizes of $H_1(n,k,t)$ and $H_2(n,k,t)$.
Every member of $H_1(n,k,t)$ contains one of the $t+2$ subsets of $[t+2]$ of size $t+1$. Hence
\[
|H_1(n,k,t)|\le (t+2)\binom{n-t-1}{k-t-1}\le (k+1)\binom{n-t-1}{k-t-1}.
\]
For $H_2(n,k,t)$, the first part is contained in the union, over the $k-t+1$ points of
$[t+1,k+1]$, of the families of $k$-sets containing $[t]$ and that point. The second part contributes
$t$ additional sets. Thus
\[
|H_2(n,k,t)|\le (k-t+1)\binom{n-t-1}{k-t-1}+t.
\]
Since $\binom{n-t-1}{k-t-1}\ge 1$, this is at most
\[
(k-t+1+t)\binom{n-t-1}{k-t-1}=(k+1)\binom{n-t-1}{k-t-1}.
\]
The result follows.
\end{proof}

We also use the cross-intersecting product theorem of Matsumoto and Tokushige \cite{MT} and
the cross-$s$-intersecting product bounds of Chen, Li, Wu, and Zhang and of Zhang and Wu \cite{CLWZ,ZW}.

\begin{theorem}[Matsumoto--Tokushige]\label{thm:MT}
Let $\A,\B\subseteq \binom{[n]}{k}$ be cross-intersecting, and assume $n\ge 2k$. Then
\[
|\A|\,|\B|\le \binom{n-1}{k-1}^2.
\]
\end{theorem}

\begin{theorem}[Chen--Li--Wu--Zhang; Zhang--Wu]\label{thm:cross}
Let $2\le s\le k$, and let $\A,\B\subseteq \binom{[n]}{k}$ be cross-$s$-intersecting. Assume either
\begin{enumerate}
    \item[(i)] $s=2$ and $n\ge 3.38k$, or
    \item[(ii)] $s\ge 3$ and $n\ge (s+1)(k-s+1)$.
\end{enumerate}
Then
\[
|\A|\,|\B|\le \binom{n-s}{k-s}^2.
\]
\end{theorem}

\section{Proof of the cubic theorem}\label{sec:cubic}

\begin{proof}[Proof of Theorem~\ref{thm:cubic}]
If $\F=\varnothing$, then $\Phi_{n,k}(\F)=0$, so assume $\F\neq \varnothing$.

Apply Lemma~\ref{lem:core} and relabel the reduced parameters again by $(n,k,\F)$. The quantity
$n-k$ is unchanged by this reduction. Since the function $x\mapsto (x^3+2x^2-x)/2$ is increasing
for integers $x\ge 1$, the numerical hypothesis remains valid after replacing $k$ by the smaller
value $k-c$. Hence we may assume that
\[
\bigcap_{A\in \F}A=\varnothing.
\]
If the reduction yields $k=0$, then $\F=\{\varnothing\}$ in the reduced problem, so the original
family is a singleton, which is a $k$-star. Thus assume from now on that $k\ge 1$.

If $\G_1=\varnothing$, then Lemma~\ref{lem:tel} gives
\[
\Phi_{n,k}(\F)=\frac{|\F|}{\binom{n}{k}}\le 1.
\]
So assume that $\G_1\neq \varnothing$. Choose $S\in \G_1$. Every member of $\F$ meets $S$, so by the
union bound,
\[
|\F|\le |S|\binom{n-1}{k-1}=k\binom{n-1}{k-1},
\]
and therefore
\[
\frac{|\F|}{\binom{n}{k}}\le \frac{k^2}{n}.
\]

Next fix $t\in \{1,\dots,k-1\}$. If $\G_t$ is trivial, then all members of $\G_t$ contain some fixed
$t$-set $T$. Since the common core of $\F$ is empty, there exists $B\in \F$ missing at least one element
of $T$. Put $D:=B\setminus T$. Then $|D|\le k$. For every $A\in \G_t$, we have $|A\cap B|\ge t$,
while $B$ contains at most $t-1$ elements of $T$, so $A$ must meet $D$. Hence
\[
\G_t\subseteq \bigcup_{x\in D}\{A\in \binom{[n]}{k}:T\cup\{x\}\subseteq A\},
\]
and thus
\[
|\G_t|\le |D|\binom{n-t-1}{k-t-1}\le k\binom{n-t-1}{k-t-1}.
\]
If instead $\G_t$ is nontrivial, then
\[
n\ge \frac{k^3+2k^2+k}{2}>k^2\ge (t+1)(k-t+1),
\]
so Corollary~\ref{cor:coarse} gives
\[
|\G_t|\le (k+1)\binom{n-t-1}{k-t-1}.
\]
Thus in every case,
\[
|\G_t|\le (k+1)\binom{n-t-1}{k-t-1}
\qquad (1\le t\le k-1).
\]

Applying Lemma~\ref{lem:tel}, we obtain
\begin{align*}
\Phi_{n,k}(\F)
&< \frac{k^2}{n}+(k+1)\sum_{t=1}^{k-1}\frac{\binom{n-t-1}{k-t-1}}{\binom{n-t}{k-t}} \\
&= \frac{k^2}{n}+(k+1)\sum_{t=1}^{k-1}\frac{k-t}{n-t} \\
&< \frac{k^2}{n-k}+\frac{k+1}{n-k}\sum_{t=1}^{k-1}(k-t) \\
&= \frac{k^2}{n-k}+\frac{(k+1)k(k-1)}{2(n-k)} \\
&= \frac{k^3+2k^2-k}{2(n-k)} \\
&\le 1.
\end{align*}
This proves the upper bound.

The same argument also yields the equality cases. Equality is impossible when $\G_1\neq \varnothing$,
so equality can occur only when $\G_1=\varnothing$. In that case,
\[
\Phi_{n,k}(\F)=\frac{|\F|}{\binom{n}{k}},
\]
so equality forces $|\F|=\binom{n}{k}$ in the reduced problem. Hence the reduced family is the full
level and so the original family is a $c$-star.
\end{proof}

\section{Proof of the quadratic theorem}\label{sec:quadratic}

\begin{lemma}\label{lem:layer}
Assume that $\bigcap_{A\in \F}A=\varnothing$. Fix $t$ with $1\le t\le k-1$, and assume that
\[
n>(t+1)(k-t+1).
\]
Then
\[
\frac{|\G_t|}{\binom{n-t}{k-t}}
\le
\max\!\left\{
1-\frac{(n-k)\bigl((n-k)-t(k-t)-1\bigr)}{(n-t)(n-t-1)},\
1-\frac{\binom{n-k-1}{k-t}}{\binom{n-t}{k-t}}+\frac{t}{\binom{n-t}{k-t}}
\right\}.
\]
\end{lemma}

\begin{proof}
If $\G_t=\varnothing$, there is nothing to prove.

Assume first that $\G_t$ is nontrivial. Then Theorem~\ref{thm:AK} gives
\[
|\G_t|\le
\max\!\left\{
(t+2)\binom{n-t-2}{k-t-1}+\binom{n-t-2}{k-t-2},\
\binom{n-t}{k-t}-\binom{n-k-1}{k-t}+t
\right\}.
\]
Dividing by $\binom{n-t}{k-t}$ yields the second entry of the maximum immediately. For the first
entry, note that
\[
\frac{\binom{n-t-2}{k-t-1}}{\binom{n-t}{k-t}}
=
\frac{(k-t)(n-k)}{(n-t)(n-t-1)},
\qquad
\frac{\binom{n-t-2}{k-t-2}}{\binom{n-t}{k-t}}
=
\frac{(k-t)(k-t-1)}{(n-t)(n-t-1)}.
\]
Hence
\begin{align*}
\frac{(t+2)\binom{n-t-2}{k-t-1}+\binom{n-t-2}{k-t-2}}{\binom{n-t}{k-t}} 
&=
\frac{(t+2)(k-t)(n-k)+(k-t)(k-t-1)}{(n-t)(n-t-1)} \\
&=
1-\frac{(n-k)\bigl((n-k)-t(k-t)-1\bigr)}{(n-t)(n-t-1)}.
\end{align*}
This proves the lemma when $\G_t$ is nontrivial.

It remains to treat the case when $\G_t$ is trivial. Then there is a fixed $t$-set $T$ contained in
every member of $\G_t$. Since the common core of $\F$ is empty, we can choose $B\in \F$ with
$T\not\subseteq B$. Write $r:=|B\cap T|\le t-1$. Every $A\in \G_t$ contains $T$ and satisfies
$|A\cap B|\ge t$, so $A\setminus T$ must contain at least $t-r$ elements of $B\setminus T$.

Choose a set $E\subseteq B\setminus T$ of size $k-t+1$; this is possible because
$|B\setminus T|=k-r\ge k-t+1$. If $X\subseteq [n]\setminus T$ has size $k-t$ and is disjoint from $E$, then
\[
|(B\setminus T)\setminus E|=(k-r)-(k-t+1)=t-r-1,
\]
so $X$ contains at most $t-r-1$ elements of $B\setminus T$. Therefore $T\cup X\notin \G_t$.
The number of such forbidden $X$ is
\[
\binom{n-t-(k-t+1)}{k-t}=\binom{n-k-1}{k-t}.
\]
Hence
\[
|\G_t|\le \binom{n-t}{k-t}-\binom{n-k-1}{k-t},
\]
which is stronger than the second entry of the maximum.
\end{proof}

\begin{proposition}\label{prop:large}
There exists an integer $k^*$ such that the following holds. Let $k\ge k^*$, and let
$\F\subseteq \binom{[n]}{k}$ satisfy
\[
\bigcap_{A\in \F}A=\varnothing,
\qquad
\G_1\neq\varnothing,
\qquad
n-k\ge \frac{k^2}{4}+4k.
\]
Then
\[
\Phi_{n,k}(\F)<1.
\]
\end{proposition}

\begin{proof}

Since $n\ge 3.38k$ and
\[
(s+1)(k-s+1)\le \frac{(k+2)^2}{4}<\frac{k^2}{4}+5k\le n,
\]
 Theorems~\ref{thm:AK}, \ref{thm:MT}, and~\ref{thm:cross} apply in every place where we will use them.

Choose $m\in\{1,\dots,k-1\}$ such that
\[
M:=\frac{|\G_m|}{\binom{n-m}{k-m}}=
\max_{1\le t\le k-1}\frac{|\G_t|}{\binom{n-t}{k-t}},
\]
and write
\[
\rho:=\frac{k-1}{n-k+1}.
\]
Since $n-k\ge k^2/4+4k$, we have $\rho\le 4/k$.

If $d\ge 1$ and $m+d\le k-1$, then $\G_m$ and $\G_{m+d}$ are cross-$(m+d)$-intersecting, so
\[
|\G_m|\,|\G_{m+d}|\le \binom{n-m-d}{k-m-d}^2.
\]
Dividing by $\binom{n-m}{k-m}\binom{n-m-d}{k-m-d}$ gives
\[
M\cdot \frac{|\G_{m+d}|}{\binom{n-m-d}{k-m-d}}
\le
\frac{\binom{n-m-d}{k-m-d}}{\binom{n-m}{k-m}}
=
\prod_{r=0}^{d-1}\frac{k-m-r}{n-m-r}.
\]
For each $r\in\{0,\dots,d-1\}$, the numerator is at most $k-1$ and the denominator is at least
$n-k+1$, so every factor is at most $\rho$. Hence the product is at most $\rho^d$.
Similarly, if $d\ge 1$ and $m-d\ge 1$, then $\G_{m-d}$ and $\G_m$ are cross-$m$-intersecting, so
\[
\frac{|\G_{m-d}|}{\binom{n-m+d}{k-m+d}}\cdot M
\le
\frac{\binom{n-m}{k-m}}{\binom{n-m+d}{k-m+d}}
=
\prod_{r=1}^{d}\frac{k-m+r}{n-m+r}.
\]
Here again each numerator is at most $k-1$ and each denominator is at least $n-k+1$, so every
factor is at most $\rho$. Hence this product is at most $\rho^d$.
Therefore, for every $d$,
\[
\frac{|\G_{m\pm d}|}{\binom{n-m\mp d}{k-m\mp d}}
\le
\min\!\left\{M,\frac{\rho^d}{M}\right\}
\le \rho^{d/2}.
\]
 The families $\F$ and $\G_m$ are cross-$m$-intersecting. By Theorems~\ref{thm:MT} and~\ref{thm:cross}, 
$|\F|\,|\G_m|\le \binom{n-m}{k-m}^2$. Thus,
\[
\frac{|\F|}{\binom{n}{k}}\,M
\le
\frac{\binom{n-m}{k-m}}{\binom{n}{k}}
\le \frac{k}{n}.
\]
Choose $S\in \G_1$. Every member of $\F$ meets $S$, so all $k$-sets disjoint from $S$ are excluded. Thus
\[
\frac{|\F|}{\binom{n}{k}}
\le 1-\frac{\binom{n-k}{k}}{\binom{n}{k}}.
\]
For every $1\le t\le k-1$, Lemma~\ref{lem:layer} gives
\[
\frac{|\G_t|}{\binom{n-t}{k-t}}
\le
\max\!\left\{
1-\frac{(n-k)\bigl((n-k)-t(k-t)-1\bigr)}{(n-t)(n-t-1)},\
1-\frac{\binom{n-k-1}{k-t}}{\binom{n-t}{k-t}}+\frac{t}{\binom{n-t}{k-t}}
\right\}.
\]

We claim that for all sufficiently large $k$,
\[
\frac{|\F|}{\binom{n}{k}}\le 1-e^{-9},
\qquad
1-\frac{\binom{n-k-1}{k-t}}{\binom{n-t}{k-t}}+\frac{t}{\binom{n-t}{k-t}}\le 1-\frac{e^{-9}}{2}
\quad (1\le t\le k-1).
\]
Indeed, write $y:=n-k$. Then $y\ge k^2/4+4k$, and
\[
\frac{n-2k}{n}=\frac{y-k}{y+k}.
\]
The function $x\mapsto (x-k)/(x+k)$ is increasing, so
\[
\frac{n-2k}{n}\ge
\frac{\frac{k^2}{4}+3k}{\frac{k^2}{4}+5k}.
\]
Since the $k$th power of the right-hand side tends to $e^{-8}$, we have
\[
\left(\frac{n-2k}{n}\right)^k\ge e^{-9}
\]
for all sufficiently large $k$.
Now
\[
\frac{\binom{n-k}{k}}{\binom{n}{k}}
=
\prod_{j=0}^{k-1}\frac{n-k-j}{n-j}
\ge
\left(\frac{n-2k}{n}\right)^k,
\]
so $|\F|/\binom{n}{k}\le 1-e^{-9}$.
Also,
\[
\frac{\binom{n-k-1}{k-t}}{\binom{n-t}{k-t}}
=
\prod_{i=1}^{k-t}\frac{n-k-i}{n-k+i},
\]
and every factor on the right is at least $(n-2k)/n$, so
\[
\frac{\binom{n-k-1}{k-t}}{\binom{n-t}{k-t}}
\ge
\left(\frac{n-2k}{n}\right)^{k-t}
\ge
\left(\frac{n-2k}{n}\right)^k
\ge e^{-9}.
\]
Finally, if $k-t=1$, then
\[
\frac{t}{\binom{n-t}{k-t}}=\frac{t}{n-t}\le \frac{k}{n-k+1},
\]
and if $k-t\ge 2$, then
\[
\frac{t}{\binom{n-t}{k-t}}
\le
\frac{k}{\binom{n-t}{2}}
\le
\frac{k}{\binom{n-k+2}{2}}.
\]
Both bounds tend to $0$ as $k\to\infty$, so
\[
\frac{t}{\binom{n-t}{k-t}}\le \frac{e^{-9}}{2}
\]
for every $t$ once $k$ is large enough. This proves the claim.

We consider three cases.

\medskip
\noindent\emph{Case 1: $M\le \sqrt{\rho}$.}
Then
\[
\sum_{t=1}^{k-1}\frac{|\G_t|}{\binom{n-t}{k-t}}
\le M+2\sum_{d\ge 1}\rho^{d/2}
\le \sqrt{\rho}+\frac{2\sqrt{\rho}}{1-\sqrt{\rho}}.
\]
Since $\rho\le 4/k$, the right-hand side tends to $0$. By Lemma~\ref{lem:tel},
\[
\Phi_{n,k}(\F)
\le
\frac{|\F|}{\binom{n}{k}}+
\sum_{t=1}^{k-1}\frac{|\G_t|}{\binom{n-t}{k-t}}
\le
1-e^{-9}+\sqrt{\rho}+\frac{2\sqrt{\rho}}{1-\sqrt{\rho}}<1
\]
for all sufficiently large $k$.

\medskip
\noindent\emph{Case 2: $\sqrt{\rho}<M\le 1-\frac{e^{-9}}{2}$.}
Now $\rho^d/M\le M$ for every $d\ge 1$, so
\[
\sum_{t=1}^{k-1}\frac{|\G_t|}{\binom{n-t}{k-t}}
\le M+\frac{2\rho}{M(1-\rho)}.
\]
Combining this with Lemma~\ref{lem:tel} and
$\frac{|\F|}{\binom{n}{k}}M\le \frac{k}{n}$, we get
\[
\Phi_{n,k}(\F)
\le
M+\frac{k}{nM}+\frac{2\rho}{M(1-\rho)}
\le
1-\frac{e^{-9}}{2}+\frac{k}{n\sqrt{\rho}}+\frac{2\sqrt{\rho}}{1-\rho}.
\]
Write $y:=n-k$. Then
\[
\frac{k}{n\sqrt{\rho}}
=
\frac{k}{y+k}\sqrt{\frac{y+1}{k-1}}.
\]
Since $y\ge k^2/4+4k>k-2$, the function $x\mapsto k\sqrt{x+1}/(x+k)$ is decreasing, so
\[
\frac{k}{n\sqrt{\rho}}
\le
\frac{k}{\frac{k^2}{4}+5k}\sqrt{\frac{\frac{k^2}{4}+4k+1}{k-1}}
=O(k^{-1/2}).
\]
Also $\rho\le 4/k$, so
\[
\frac{2\sqrt{\rho}}{1-\rho}=O(k^{-1/2}).
\]
Hence $\Phi_{n,k}(\F)<1$ for all sufficiently large $k$.

\medskip
\noindent\emph{Case 3: $M>1-\frac{e^{-9}}{2}$.}
The maximizing layer cannot come from the second entry of the maximum above, because that
entry is at most $1-\frac{e^{-9}}{2}$ while $M>1-\frac{e^{-9}}{2}$. Therefore
\[
1-M\ge \frac{(n-k)\bigl((n-k)-m(k-m)-1\bigr)}{(n-m)(n-m-1)}.
\]
Write the right-hand side as
\[
\frac{n-k}{n-m}\cdot \frac{(n-k)-m(k-m)-1}{n-m-1}.
\]
Both factors are increasing functions of $n-k$. Set $x_0:=\frac{k^2}{4}+4k$. Then $n-k\ge x_0$, so
\[
\frac{n-k}{n-m}\cdot \frac{(n-k)-m(k-m)-1}{n-m-1}
\ge
\frac{x_0}{x_0+k-m}\cdot \frac{x_0-m(k-m)-1}{x_0+k-m-1}.
\]
Now $k-m\le k$, while $m(k-m)\le k^2/4$, so
\[
\frac{x_0}{x_0+k-m}\ge \frac{x_0}{x_0+k},
\qquad
x_0-m(k-m)-1\ge 4k-1,
\qquad
x_0+k-m-1\le x_0+k-1.
\]
Therefore
\[
1-M\ge
\frac{\frac{k^2}{4}+4k}{\frac{k^2}{4}+5k}
\cdot
\frac{4k-1}{\frac{k^2}{4}+5k-1}.
\]
Now
\[
\frac{\frac{k^2}{4}+4k}{\frac{k^2}{4}+5k}=1+O(k^{-1}),
\qquad
\frac{4k-1}{\frac{k^2}{4}+5k-1}=\frac{16+o(1)}{k},
\]
so
\[
1-M\ge \frac{16+o(1)}{k}.
\]
On the other hand, using $M>1-\frac{e^{-9}}{2}$, we obtain
\[
\frac{|\F|}{\binom{n}{k}}+
\sum_{t\ne m}\frac{|\G_t|}{\binom{n-t}{k-t}}
\le
\frac{1}{1-\frac{e^{-9}}{2}}\left(\frac{k}{n}+\frac{2\rho}{1-\rho}\right).
\]
Now
\[
\frac{k}{n}\le \frac{k}{\frac{k^2}{4}+5k}=\frac{4+o(1)}{k},
\qquad
\frac{2\rho}{1-\rho}=\frac{2(k-1)}{n-2k+2}
\le \frac{2(k-1)}{\frac{k^2}{4}+3k+2}=\frac{8+o(1)}{k}.
\]
Therefore
\[
\frac{|\F|}{\binom{n}{k}}+
\sum_{t\ne m}\frac{|\G_t|}{\binom{n-t}{k-t}}
\le
\frac{12+o(1)}{(1-\frac{e^{-9}}{2})k}.
\]
Since
\[
\frac{12}{1-\frac{e^{-9}}{2}}<16,
\]
the deficit $1-M$ is eventually larger than this error term. Therefore
\[
\Phi_{n,k}(\F)
\le
M+\frac{|\F|}{\binom{n}{k}}+
\sum_{t\ne m}\frac{|\G_t|}{\binom{n-t}{k-t}}<1.
\]

The three cases prove the proposition.
\end{proof}

\begin{lemma}\label{lem:fixed}
For each integer $r\ge 1$ there exists an integer $N_r$ such that the following holds.
Let $\F\subseteq \binom{[n]}{r}$ satisfy
\[
\bigcap_{A\in \F}A=\varnothing,
\qquad
\G_1\neq\varnothing,
\qquad
n-r\ge N_r.
\]
Then
\[
\Phi_{n,r}(\F)<1.
\]
\end{lemma}

\begin{proof}
Choose $N_r$ so large that whenever $n-r\ge N_r$, we have
\[
n>(t+1)(r-t+1)\qquad (1\le t\le r-1)
\]
and also
\[
\frac{r^2}{n}+\frac{(r-1)(r+1)^2}{n-r+1}<1.
\]
This is possible because $r$ is fixed and the left-hand side tends to $0$ as $n\to\infty$.

Choose $S\in \G_1$. Every member of $\F$ meets $S$, so
\[
\frac{|\F|}{\binom{n}{r}}
\le
r\,\frac{\binom{n-1}{r-1}}{\binom{n}{r}}
=
\frac{r^2}{n}.
\]
Now fix $t$ with $1\le t\le r-1$. Since $n-r\ge N_r$, Lemma~\ref{lem:layer} applies.
For the first entry of the maximum,
\begin{align*}
1-\frac{(n-r)\bigl((n-r)-t(r-t)-1\bigr)}{(n-t)(n-t-1)}
&=
\frac{(t+2)(r-t)(n-r)+(r-t)(r-t-1)}{(n-t)(n-t-1)} \\
&\le
\frac{(t+2)(r-t)}{n-r+1}+\frac{r-t}{n-r+1} \\
&\le \frac{(r+1)^2}{n-r+1}.
\end{align*}
For the second entry, using $1-\prod(1-a_i)\le \sum a_i$ for $0\le a_i\le 1$,
\begin{align*}
1-\frac{\binom{n-r-1}{r-t}}{\binom{n-t}{r-t}}+\frac{t}{\binom{n-t}{r-t}}
&\le
\sum_{i=1}^{r-t}\frac{2i}{n-r+i}+\frac{r}{n-r+1} \\
&\le
\frac{(r-t)(r-t+1)}{n-r+1}+\frac{r}{n-r+1} \\
&\le \frac{(r+1)^2}{n-r+1}.
\end{align*}
Hence
\[
\frac{|\G_t|}{\binom{n-t}{r-t}}\le \frac{(r+1)^2}{n-r+1}
\qquad (1\le t\le r-1).
\]
Applying Lemma~\ref{lem:tel},
\[
\Phi_{n,r}(\F)
\le
\frac{|\F|}{\binom{n}{r}}+
\sum_{t=1}^{r-1}\frac{|\G_t|}{\binom{n-t}{r-t}}<1
\]
by the choice of $N_r$.
\end{proof}

\begin{proof}[Proof of Theorem~\ref{thm:quadratic}]
Let $k^*$ be given by Proposition~\ref{prop:large}. For each $1\le r<k^*$, let $N_r$ be given by
Lemma~\ref{lem:fixed}, and define
\[
D:=\max_{1\le r<k^*}\left(N_r-\frac{r^2}{4}-4r\right).
\]
Increasing $D$ if necessary, we may also assume that $D\ge 0$.

Now let $\F\subseteq \binom{[n]}{k}$ and assume that
\[
n-k\ge \frac{k^2}{4}+4k+D.
\]
If $\F=\varnothing$, then $\Phi_{n,k}(\F)=0$, so there is nothing to prove.

Apply Lemma~\ref{lem:core}. Thus we may remove the common core and replace $(n,k,\F)$ by
$(n',k',\F')$ such that
\[
\Phi_{n,k}(\F)=\Phi_{n',k'}(\F'),
\qquad
n'-k'=n-k,
\qquad
\bigcap_{A\in \F'}A=\varnothing.
\]

We first prove the bound $\Phi_{n,k}(\F)\le 1$.
If $k'=0$, then $\F'=\{\varnothing\}$ and $\Phi_{n',0}(\F')=1$.
Assume from now on that $k'\ge 1$.

If $\G'_1=\varnothing$, then $i_{\F'}(A)=0$ for every $A\in \F'$, and hence
\[
\Phi_{n',k'}(\F')=\frac{|\F'|}{\binom{n'}{k'}}\le 1.
\]
So only the case $\G'_1\neq\varnothing$ remains.

If $k'\ge k^*$, then
\[
n'-k'=n-k\ge \frac{k^2}{4}+4k+D\ge \frac{(k')^2}{4}+4k',
\]
so Proposition~\ref{prop:large} gives
\[
\Phi_{n',k'}(\F')<1.
\]
If instead $1\le k'<k^*$, then by the definition of $D$,
\[
n'-k'=n-k\ge \frac{k^2}{4}+4k+D\ge \frac{(k')^2}{4}+4k'+D\ge N_{k'},
\]
so Lemma~\ref{lem:fixed} gives
\[
\Phi_{n',k'}(\F')<1.
\]
Thus $\Phi_{n,k}(\F)\le 1$ in all cases.

We now determine the equality cases. Every case with $\G'_1\neq\varnothing$ gives a strict inequality,
so equality can occur only if $k'=0$ or $\G'_1=\varnothing$.

If $k'=0$, then the original family consists of one $k$-set, so it is a $k$-star.
Assume now that $k'\ge 1$ and $\G'_1=\varnothing$. Then
\[
\Phi_{n',k'}(\F')=\frac{|\F'|}{\binom{n'}{k'}},
\]
so equality holds if and only if
\[
|\F'|=\binom{n'}{k'},
\qquad\text{that is,}\qquad
\F'=\binom{[n']}{k'},
\]
and thus $\F$ is a $c$-star.
\end{proof}

\section{Quadratic-order counterexamples and sharpness}\label{sec:sharpness}

Recall the Ahlswede--Khachatrian families~\cite{AK}:
\[
J_t(n,k):=\{A\in \binom{[n]}{k}:|A\cap [t+2]|\ge t+1\}.
\]
We have
\[
|J_t(n,k)|=(t+2)\binom{n-t-2}{k-t-1}+\binom{n-t-2}{k-t-2},
\]
and comparing this with
\[
\binom{n-t}{k-t}
=
\binom{n-t-2}{k-t}+2\binom{n-t-2}{k-t-1}+\binom{n-t-2}{k-t-2}
\]
shows that
\[
|J_t(n,k)|-\binom{n-t}{k-t}
=
\frac{(t+1)(k-t+1)-n}{k-t}\binom{n-t-2}{k-t-1}.
\]
Hence $|J_t(n,k)|>\binom{n-t}{k-t}$ and thus $\Phi_{n,k}(J_t(n,k))>1$ when $n<(t+1)(k-t+1)$.
 The following proposition shows that our localized inequality is violated for some $n$ beyond this range.

\begin{proposition}\label{prop:Jt}
Let $1\le t\le k-2$. Then
\[
\Phi_{n,k}(J_t(n,k))
\ge
1+\frac{(n-k)\bigl((k-t)(t+1)-(n-k)\bigr)}{(n-t)(n-t-1)}.
\]
In particular, if
\[
k<n<k+(k-t)(t+1),
\]
then
\[
\Phi_{n,k}(J_t(n,k))>1.
\]
\end{proposition}

\begin{proof}
Partition $J_t(n,k)$ into the sets of the following two types:
\[
\mathcal{X}:=\{A\in J_t(n,k):[t+2]\subseteq A\},
\qquad
\mathcal{Y}:=J_t(n,k)\setminus \mathcal{X}.
\]
Then
\[
\Phi_{n,k}(J_t(n,k))\ge \frac{|\mathcal{Y}|}{\binom{n-t}{k-t}}+\frac{|\mathcal{X}|}{\binom{n-t-1}{k-t-1}}.
\]
Now
\[
|\mathcal{X}|=\binom{n-t-2}{k-t-2},
\qquad
|\mathcal{Y}|=(t+2)\binom{n-t-2}{k-t-1}.
\]
Using
\[
\frac{\binom{n-t-2}{k-t-2}}{\binom{n-t-1}{k-t-1}}=\frac{k-t-1}{n-t-1},
\qquad
\frac{\binom{n-t-2}{k-t-1}}{\binom{n-t}{k-t}}=\frac{(k-t)(n-k)}{(n-t)(n-t-1)},
\]
we obtain
\[
\Phi_{n,k}(J_t(n,k))
\ge
(t+2)\frac{(k-t)(n-k)}{(n-t)(n-t-1)}+\frac{k-t-1}{n-t-1}.
\]
Simplifying gives the stated bound.
\end{proof}

\begin{corollary}\label{cor:quadratic-sharp}
Let $k\ge 3$ and 
\[
k<n<k+\left\lceil\frac{k}{2}\right\rceil\left(\left\lfloor\frac{k}{2}\right\rfloor+1\right),
\]
then there exists a family $\F\subseteq \binom{[n]}{k}$ such that $\Phi_{n,k}(\F)>1$. 
\end{corollary}

\begin{proof}
For $t=\lfloor k/2\rfloor$, we have
\[
(k-t)(t+1)=\left\lceil\frac{k}{2}\right\rceil\left(\left\lfloor\frac{k}{2}\right\rfloor+1\right).
\]
The result follows from Proposition~\ref{prop:Jt}.
\end{proof}

We conjecture that the family in Corollary~\ref{cor:quadratic-sharp} determines the minimum value 
$n_0$ for which the conclusion of Theorems~\ref{thm:cubic} and~\ref{thm:quadratic} holds for all $n \ge n_0$.

\begin{conjecture}
Let $k \ge 3$ and $\F\subseteq \binom{[n]}{k}$, and assume that
\[
n\ge k+\left\lceil\frac{k}{2}\right\rceil\left(\left\lfloor\frac{k}{2}\right\rfloor+1\right),
\]
Then
\[
\sum_{A\in \F}\frac{1}{\binom{n-i_{\F}(A)}{k-i_{\F}(A)}}\le 1.
\]
\end{conjecture}

\section{Hilton-type consequences}\label{sec:hilton}

We now prove our Hilton-type result, Theorem~\ref{thm:multi-hilton}, using the localized results Theorems~\ref{thm:cubic} and~\ref{thm:quadratic}.

\begin{proof}[Proof of Theorem~\ref{thm:multi-hilton}]
Put
\[
\mathcal{U}:=\bigcup_{i=1}^k\bigcup_{j=1}^{m_i}\F_{i,j}.
\]
For $A\in \mathcal{U}$, let
\[
r(A):=\left|\{(i,j): A\in \F_{i,j}\}\right|.
\]
Then
\[
\sum_{i=1}^k\sum_{j=1}^{m_i} |\F_{i,j}|=\sum_{A\in \mathcal{U}} r(A).
\]
So it is enough to prove that
\[
r(A)\le M_{i_{\mathcal{U}}(A)}
\qquad (A\in \mathcal{U}).
\]
Fix $A\in \mathcal{U}$ and write $s:=i_{\mathcal{U}}(A)$. If $r(A)=1$, then $r(A)=1\le M_s$.
Assume now that $r(A)\ge 2$, and let $A\in \F_{p,\ell}$. Choose a  family
$\F_{q,\ell'}$ containing $A$ with $(q,\ell') \neq (p,\ell)$.

We claim that $s\ge p$. Let $B\in \mathcal{U}$ be arbitrary. If $B\notin \F_{p,\ell}$, then
$B\in \F_{i',j'}$ for some $(i',j')\ne (p,\ell)$, and so
\[
|A\cap B|\ge \max\{p,i'\}\ge p.
\]
If instead $B\in \F_{p,\ell}$, then using $A\in \F_{q,\ell'}$ and
$(p,\ell)\ne (q,\ell')$, we obtain
\[
|A\cap B|\ge \max\{p,q\}\ge p.
\]
Thus $|A\cap B|\ge p$ for every $B\in \mathcal{U}$, and hence $s=i_{\mathcal{U}}(A)\ge p$.
Since $\F_{p,\ell}$ was an arbitrary family containing $A$, every such family has first index at most $s$.
Therefore
\[
r(A)\le \sum_{r=1}^s m_r\le M_s.
\]
Put
\[
\Lambda:=\max_{0\le s\le k} M_s\binom{n-s}{k-s}.
\]
Then for every $A\in \mathcal{U}$,
\[
M_{i_{\mathcal{U}}(A)}\le \Lambda\binom{n-i_{\mathcal{U}}(A)}{k-i_{\mathcal{U}}(A)}^{-1}.
\]
Hence
\[
\sum_{i=1}^k\sum_{j=1}^{m_i} |\F_{i,j}|
\le
\sum_{A\in \mathcal{U}} M_{i_{\mathcal{U}}(A)}
\le \Lambda\,\Phi_{n,k}(\mathcal{U}).
\]
If $n\ge k^2/4+5k+D$, then Theorem~\ref{thm:quadratic} gives $\Phi_{n,k}(\mathcal{U})\le 1$.
The cubic case is identical, using Theorem~\ref{thm:cubic}. In either case, the displayed bound yields
the stated inequality.
\end{proof}

\section{Use of automated tools} Candidate arguments were explored using ChatGPT. The final arguments, proofs, and conclusions in the manuscript were written and independently verified by the author.

\end{document}